\documentclass[12pt, reqno]{amsart}
\usepackage{graphicx}
\usepackage{subfigure}
\usepackage{latexsym}
\usepackage{amsmath}
\usepackage{amssymb}
\usepackage{mathrsfs}
\usepackage{color}

 \newtheorem{theorem}{Theorem}[section]
 
 \newtheorem{Prop}[theorem]{Proposition}
 \newtheorem{Lem}[theorem]{Lemma}
 \newtheorem{Cor}[theorem]{Corollary}

 \newtheorem{prob}[theorem]{Problem}

 \numberwithin{equation}{section}

\def\Zp{\mathbb{Z}^+}

\begin{document}

\title{Lipschitz equivalence of Cantor sets and irreducibility of polynomials}

\author{Jun Jason Luo}
\address{College of Mathematics and Statistics, Chongqing University,  401331 Chongqing, China
\newline\indent Institut f\"ur Mathematik, Friedrich-Schiller-Universit\"at Jena, 07743 Jena, Germany}
\email{jasonluojun@gmail.com}

\author{Huo-Jun Ruan}
\address{ Department of Mathematics, Zhejiang University, Hangzhou 310027, China}
\email{ruanhj@zju.edu.cn}

\author{Yi-Lin Wang}
\address{College of Mathematics and Statistics,  Chongqing Uinversity, Chongqing, 401331, China}
\email{yilinwang@cqu.edu.cn}


\thanks{The research of Luo and Wang is supported in part by the NNSF of China (No. 11301322), the Fundamental and Frontier Research Project of Chongqing (No.cstc2015jcyjA00035). The research of Ruan is supported in part by NSFC (No. 11271327) and ZJNSFC (No. LR14A010001).}



\date{\today}

\begin{abstract}
In the paper, we provide an effective method for the Lipschitz equivalence of two-branch Cantor sets and three-branch Cantor sets by studying the
irreducibility of polynomials. We also find that any two Cantor sets are Lipschitz equivalent if and only if their contraction vectors are equivalent provided one of the contraction vectors is homogeneous.
\end{abstract}

\maketitle
\bigskip

\begin{section}{Introduction}
Let $E,F$ be two nonempty compact subsets of ${\mathbb R}^d$. We say that  $E$ and $F$ are 
{\it Lipschitz equivalent} and denote it by $E\sim F$ if there is a bi-Lipschitz map  $\phi$ from $E$ onto $F$, i.e., $\phi$ is a bijection and  there is a constant $C>0$ such that
$$
C^{-1}|x-y|\leq |\phi(x)-\phi(y)|\leq C|x-y| \quad \text{for all} \
  x,y\in E.$$

Lipschitz equivalence is an interesting topic in geometric measure theory and fractal geometry. It is well-known that Hausdorff dimension is a Lipschitz invariant.  Since the late 80's, there have been a lot of studies devoted to the topic (see \cite{CoPi88,FaMa89,FaMa92,LaLu12,LlMa10},\cite{RaRuWa10}-\cite{XiXi12}). A pioneer work on Cantor sets was done by Falconer and Marsh \cite{FaMa92}, where they gave some elementary conditions on the contraction ratios to determine the Lipschitz equivalence between
two {\it dust-like self-similar sets (also called Cantor sets)}. Recently, Rao, Ruan and Wang \cite{RaRuWa10} extended their result and developed several elegant algebraic criteria to characterize the Lipschitz equivalence of Cantor sets.

Let  $\{f_i\}_{i=1}^m$ with $f_i(x)=\alpha_iR_i(x+d_i)$ be an {\it iterated function system  (IFS)} on ${\mathbb{R}}^d$, where $0<\alpha_i<1$ are contraction ratios, $R_i$ are orthogonal matrices, and $d_i\in {\mathbb{R}}^d$. Then  there exists a unique nonempty compact subset $E$ \cite{Fa03} such that
\begin{equation}\label{ssset equ}
E=\bigcup_{i=1}^m f_i(E).
\end{equation}
We call such $E$ a \emph{self-similar set}, and call $E$ {\it dust-like}  if it further satisfies $f_i(E) \cap f_j(E)  = \emptyset $ for $i \not = j$. Given
${\alpha}_1,\dots,{\alpha}_m \in (0,1)$ with $\sum_{i=1}^m{\alpha}_i^d<1$, we call $\boldsymbol{\alpha}=({\alpha}_1,\dots,{\alpha}_m)$ a {\it contraction vector}, and denote by ${\mathcal{D}}(\boldsymbol{\alpha})$ the collection of all dust-like self-similar sets satisfying (\ref{ssset equ}). Clearly, all sets in
${\mathcal{D}}(\boldsymbol{\alpha})$ have the same Hausdorff dimension $s$ \cite{Fa03} which is the unique solution of
\begin{equation}\label{dim. formula}
\sum_{i=1}^m{\alpha}_i^s=1.
\end{equation}
 Moreover, any two sets in ${\mathcal{D}}(\boldsymbol{\alpha})$ are always Lipschitz equivalent. We  say ${\mathcal{D}}(\boldsymbol{\alpha})$ and ${\mathcal{D}}(\boldsymbol{\beta})$ are Lipschitz equivalent, and
often denote it by ${\mathcal{D}}(\boldsymbol{\alpha})\sim {\mathcal{D}}(\boldsymbol{\beta})$， if $E\sim F$ for some (thus for all) $E\in{\mathcal{D}}(\boldsymbol{\alpha}), F\in{\mathcal{D}}(\boldsymbol{\beta})$.

Let $E\in \mathcal D(\boldsymbol{\alpha})$ and $\Sigma^*=\bigcup_{k=0}^{\infty}\Sigma^k$  be the symbolic space representing the IFS as in \eqref{ssset equ} where  $\Sigma=\{1,\dots, m\}$ and $\Sigma^0=\emptyset$. Given ${\mathbf i}=i_1\cdots i_n\in  \Sigma^*$, we denote  $f_{\mathbf i}=f_{i_1}\circ f_{i_2}\circ\cdots\circ f_{i_n}$ and $\alpha_{\mathbf {i}}=\alpha_{i_1}\cdots\alpha_{i_n}$.  A subset $\Lambda$ of $\Sigma^*$ is called a \emph{partition} for $E$ if it satisfies $E=\bigcup_{\mathbf{i}\in\Lambda}f_{\mathbf{i}}(E)$, where the union is disjoint.

Let $\boldsymbol{\alpha}=({\alpha}_1,\dots,{\alpha}_m),\ \boldsymbol{\beta}=({\beta}_1,\dots,{\beta}_n)$ be two contraction vectors. We say $\boldsymbol{\alpha}$ is derived from $\boldsymbol{\beta}$ if there exists a partition $\Lambda=\{{\mathbf j}_i,\dots, {\mathbf j}_m\}$ such that $\boldsymbol{\alpha}=({\beta}_{{\mathbf j}_1},\dots, {\beta}_{{\mathbf j}_m})$. $\boldsymbol{\alpha}$ and $\boldsymbol{\beta}$  are called equivalent, denoted by $\boldsymbol{\alpha} \sim \boldsymbol{\beta}$,  if there exists a sequence $$\boldsymbol{\alpha}=\boldsymbol{\alpha}_1, \boldsymbol{\alpha}_2,\dots, \boldsymbol{\alpha}_N=\boldsymbol{\beta}$$ such that $\boldsymbol{\alpha}_{j+1}$ is derived from $\boldsymbol{\alpha}_j$ or vice versa for $1\leq j \leq N$. Trivially, if $\boldsymbol{\alpha} \sim \boldsymbol{\beta}$ then ${\mathcal{D}}(\boldsymbol{\alpha})\sim {\mathcal{D}}(\boldsymbol{\beta})$. A quite natural question is

\begin{prob}[Problem 1.6 in \cite{RaRuWa12}] \label{main-prob}
Find nontrivial sufficient conditions and necessary conditions on $\boldsymbol{\alpha}$ and $\boldsymbol{\beta}$ such that ${\mathcal{D}}(\boldsymbol{\alpha})\sim {\mathcal{D}}(\boldsymbol{\beta})$. In particular, is it true that ${\mathcal{D}}(\boldsymbol{\alpha})\sim {\mathcal{D}}(\boldsymbol{\beta})$ if and only if  $\boldsymbol{\alpha}\sim\boldsymbol{\beta}$?
\end{prob}

Write ${\mathbb Q}({\alpha}_1,\dots,{\alpha}_m)$ for the sub-field of $({\mathbb R},+,\times)$ formed by the rational functions of $\alpha_1,\dots,\alpha_m$, and sgp$({\alpha}_1,\dots,{\alpha}_m)$ for the sub-semigroup of $({\mathbb R}^+, \times)$ generated by ${\alpha}_1,\dots,{\alpha}_m$.  Falconer and Marsh \cite{FaMa92} provided some algebraic conditions for the above problem.

\begin{theorem}[\cite{FaMa92}]\label{Falconer necessary condition on lip.equ.}
Let $\boldsymbol{\alpha}=({\alpha}_1,\dots,{\alpha}_m), \boldsymbol{\beta}=({\beta}_1,\dots,{\beta}_n)$ be two contraction vectors. If
${\mathcal{D}}(\boldsymbol{\alpha})\sim{\mathcal{D}}(\boldsymbol{\beta})$ with common Hausdorff dimension $s$, then

\indent $(1)$  ${\mathbb Q}({\alpha}_1^s,\dots,{\alpha}_m^s)={\mathbb Q}({\beta}_1^s,\dots,{\beta}_n^s)$;

\indent  $(2)$ there exist $p,q\in \Zp$ such that
\begin{align*}
\mathrm{sgp}({\alpha}_1^p,\dots,{\alpha}_m^p) & \subset \mathrm{sgp}({\beta}_1,\dots,{\beta}_n), \\
\mathrm{sgp}({\beta}_1^q,\dots,{\beta}_n^q) & \subset
\mathrm{sgp}({\alpha}_1,\dots,{\alpha}_m).
\end{align*}
\end{theorem}

Let $\langle\boldsymbol\alpha\rangle$ denote the free abelian group of $({\mathbb R}^+, \times)$ generated by ${\alpha}_1,\dots,{\alpha}_m$. In \cite{RaRuWa10}, Rao, Ruan and Wang defined a concept of  {\it rank} for $\boldsymbol{\alpha}$, say $\text{rank}\langle\boldsymbol\alpha\rangle$, to be the cardinality of the basis for the free abelian group $\langle\boldsymbol\alpha\rangle$. They partially improved Falconer and Marsh's theorem and obtained the necessary and sufficient condition on ${\mathcal{D}}(\boldsymbol{\alpha})\sim {\mathcal{D}}(\boldsymbol{\beta})$ in particular cases:

\begin{theorem}[\cite{RaRuWa10}]\label{Raothm1}
Let $\boldsymbol{\alpha}=({\alpha}_1,\dots,{\alpha}_m),
\boldsymbol{\beta}=({\beta}_1,\dots,{\beta}_m)$ be two contraction
vectors. If $\mathrm{rank}\langle\boldsymbol\alpha\rangle=m$, then
${\mathcal{D}}(\boldsymbol{\alpha})\sim
{\mathcal{D}}(\boldsymbol{\beta})$ if and only if
$\boldsymbol{\alpha}$ is a permutation of $\boldsymbol{\beta}$.

If $m=2$ and assume that $\alpha_1\leq \alpha_2$, $\beta_1\leq \beta_2$, $\alpha_1\leq \beta_1$, then
${\mathcal{D}}(\boldsymbol{\alpha})\sim
{\mathcal{D}}(\boldsymbol{\beta})$ if and only if $\boldsymbol{\alpha}=\boldsymbol{\beta}$, or there exists a
real number $0<\lambda<1$, such that
$({\alpha}_1,{\alpha}_2)=(\lambda^5,\lambda)$ and
$({\beta}_1,{\beta}_2)=(\lambda^3,\lambda^2)$.
\end{theorem}

Moreover, if one of the two contraction vectors is homogeneous, they gave a complete characterization on the Lipschitz equivalence as below:

\begin{theorem}[\cite{RaRuWa10}]\label{Raothm2}
  Let $\boldsymbol{\alpha}=({\alpha},\dots,{\alpha})\in \mathbb{R}^m,
\boldsymbol{\beta}=({\beta}_1,\dots,{\beta}_n)\in \mathbb{R}^n$ be two contraction
vectors. If ${\mathcal{D}}(\boldsymbol{\alpha})$ and ${\mathcal{D}}(\boldsymbol{\beta})$ have the same dimension $s$. Then ${\mathcal{D}}(\boldsymbol{\alpha})\sim
{\mathcal{D}}(\boldsymbol{\beta})$ if and only if there exist $q,p_1,\ldots,p_n\in \Zp$ such that $m^{\frac{1}{q}}\in \Zp$ and $\beta_j = \alpha^{\frac{p_j}{q}}$ for all $j=1,\ldots,n$.
\end{theorem}

In this paper, by investigating the irreducibility of trinomials and quadrinomials, we provide an effective method for the Lipschitz equivalence of certain Cantor sets.

\begin{theorem}\label{main thm}
Let $\boldsymbol{\alpha}=(\lambda^a,\lambda^b,\lambda^c)$ and $\boldsymbol{\beta}=(\lambda^d,\lambda^e)$ be two contraction vectors with $0< \lambda<1,\ a,b,c,d,e\in \Zp$. If ${\mathcal{D}}(\boldsymbol{\alpha})\sim {\mathcal{D}}(\boldsymbol{\beta})$, then the polynomial $f(x)=x^a+x^b+x^c-1$ is reducible.
\end{theorem}

In some special case, we can improve the theorem to get a necessary and sufficient condition.

\begin{theorem}\label{main thm2}
Under the same assumption as above. If $a=b+c$, $\gcd(a,b,c)=1$.  Then
	$\mathcal D(\boldsymbol{\alpha})\sim\mathcal D(\boldsymbol{\beta})$  if and only if  $a=4$ with $\{b,c\}=\{1,3\}$ and $\{d,e\}=\{1,2\}$; or $a=8$ with $\{b,c\}=\{1,7\}$ and $\{d,e\}=\{1,5\}$, $\{2,3\}$.
\end{theorem}

Finally, inspired by Theorem \ref{Raothm2}, we give an affirmative answer to the later part of Problem \ref{main-prob} when $\boldsymbol{\alpha}$ (or $\boldsymbol{\beta}$) is homogeneous.

\begin{theorem}\label{main thm3}
Let $\boldsymbol{\alpha}=({\alpha},\dots,{\alpha})\in {\mathbb{R}}^m,\  \boldsymbol{\beta}=({\beta}_1,\dots,{\beta}_n)\in {\mathbb{R}}^n$ be two contraction vectors. Then ${\mathcal{D}}(\boldsymbol{\alpha})\sim {\mathcal{D}}(\boldsymbol{\beta})$ if and only if $\boldsymbol{\alpha} \sim \boldsymbol{\beta}$.
\end{theorem}

The paper is organized as follows: In section 2, we prove Theorems \ref{main thm} and \ref{main thm2} by applying the irreducibility of integer polynomials, and  also provide several easy criteria to judge the non-Lipschitz equivalence of Cantor sets.  A short proof of Theorem \ref{main thm3} will be included in Section 3.
\end{section}

\begin{section}{Irreducibility of polynomials}

The irreducibility of polynomials is a classical subject and there are lots of related works in the literature  (please refer to \cite{Du04,FiJo06,Lj60,Mi85}).  In this section, we  recall some results on the irreducibility of  certain trinomials and quadrinomials. Then we establish the relationship between the irreducibility of polynomials and Lipschitz equivalence of Cantor sets.

\begin{Prop}[\cite{Lj60}]\label{prop1}
Let $a\geq 2b >0$. Write $a=a_1\ell, b=b_1\ell$, where
$\ell=\gcd(a,b)$. Then the polynomial $g(x)=x^a+\epsilon
x^b+\delta,\  \epsilon, \delta\in\{1,-1\}$ is irreducible unless
$a_1+b_1=0 \ (\text{mod}\ 3)$ and one of the following three
conditions holds: $a_1,b_1$ are both odd and $\epsilon=1$; $a_1$ is
even and $\delta=1$; $b_1$ is even and $\epsilon= \delta$.

In any of these exceptional cases, $g(x)$ is the product of the
polynomial $x^{2\ell}+\epsilon^{b_1}\delta^{a_1}x^\ell+1$ and a
second irreducible polynomial.
\end{Prop}

\begin{Prop}[\cite{Mi85}]\label{prop3.2}
	Suppose that $f(x)$ is a polynomial over the rationals of the form
	$$f(x)=x^a+\epsilon_1x^b+\epsilon_2x^c+\epsilon_3,$$
	where $a>b>c>0$ and $\epsilon_1,\epsilon_2,\epsilon_3\in\{1,-1\}$. Let $f(x)=A(x)B(x)$ where every root of $A(x)$ and no root of $B(x)$ is a root of unity. Then $A(x)$ is the greatest common divisor of $f(x)$ and $f^*(x)$, where $f^*(x)$ denote the reciprocal polynomial $f^*(x)=x^af(x^{-1})$. The second factor $B(x)$ is irreducible except when $f(x)$ is one of the following four forms:
	\begin{align}
		& x^{8r}+x^{7r}+x^{r}-1=(x^{2r}+1)(x^{3r}+x^{2r}-1)(x^{3r}-x^{r}+1) \label{3.1}\\
		& x^{8r}-x^{7r}-x^{r}-1=(x^{2r}+1)(x^{3r}-x^{2r}+1)(x^{3r}-x^{r}-1) \label{3.2} \\
		& x^{8r}+x^{4r}+x^{2r}-1=(x^{2r}+1)(x^{3r}+x^{2r}-1)(x^{3r}-x^{2r}+1) \label{3.3}\\
		& x^{8r}+x^{6r}+x^{4r}-1=(x^{2r}+1)(x^{3r}-x^{r}-1)(x^{3r}-x^{r}+1) \label{3.4}
	\end{align}
	In above cases, the factors of degree $3r$ are irreducible.
\end{Prop}

\begin{Prop}[\cite{Lj60}]\label{prop3.3}
	If $a=a_1t$, $b=b_1t$, $c=c_1t$, and $\gcd(a_1, b_1, c_1)=1$, $\gcd(a_1, b_1-c_1)=t_1$, $\gcd(b_1, a_1-c_1)=t_2$, $\gcd(c_1, a_1-b_1)=t_3$;  $\epsilon_1,\epsilon_2,\epsilon_3\in\{1,-1\}$ then all possible roots of unity of $f(x)=x^a+\epsilon_1x^b+\epsilon_2x^c+\epsilon_3$ are simple zeros, which are to be found among the zeros of
	\begin{align}\label{3.2}
		x^{tt_1}=\pm1,\;\;\;x^{tt_2}=\pm1,\;\;\;x^{tt_3}=\pm1.
	\end{align}
\end{Prop}

\begin{theorem}\label{mainthm1}
Let $\boldsymbol{\alpha}=(\lambda^a,\lambda^b,\lambda^c)$ and
$\boldsymbol{\beta}=(\lambda^d,\lambda^e)$ be two contraction
vectors with $0< \lambda<1,\ a,b,c,d,e\in \Zp$. If
${\mathcal{D}}(\boldsymbol{\alpha})\sim
{\mathcal{D}}(\boldsymbol{\beta})$, then the polynomial
$f(x)=x^a+x^b+x^c-1$ is reducible.
\end{theorem}

\begin{proof}
If $d=e$,  then from Theorem~\ref{Raothm2}, we have $d|a,\ d|b, \ d|c.$  Hence we need only to consider the case that $\boldsymbol{\alpha}=(\lambda^a,\lambda^b,\lambda^c)$ and $\boldsymbol{\beta}=(\lambda,\lambda)$ by letting $d=e=1$. If $D(\boldsymbol{\alpha})\sim D(\boldsymbol{\beta})$, it follows from (\ref{dim. formula})
that the common Hausdorff dimension $s$ satisfies
$$2\lambda^s=1\quad \text{and}\quad\lambda^{sa}+\lambda^{sb}+\lambda^{sc}=1.$$  This implies that $2^{-a}+2^{-b}+2^{-c}=1$, and then $(a,b,c)=(1,2,2)$ or its
permutations. It concludes that $f(x)=2x^2+x-1$,  which is reducible.

Similarly, if $a=b=c$, it suffices to consider the case $\boldsymbol{\alpha}=(\lambda,\lambda,\lambda)$ and $\boldsymbol{\beta}=(\lambda^d,\lambda^e)$ by letting $a=b=c=1$. The Hausdorff dimension $s$ satisfies $$3\lambda^s=1\quad \text{and} \quad \lambda^{sd}+\lambda^{se}=1,$$ which is impossible since
$3^{-d}+3^{-e}<1$ for any $d,e\in \Zp$. Hence this case does not occur.

In the sequel, without loss of generality, we assume that $a\geq b> c$ and $d>e$. Let $g(x)= x^d+x^e-1$. Then $g(x)$ and $f(x)$ have the
same root $\lambda^s$. Obviously, if $g(x)$ is irreducible, then $g(x)|f(x)$ so that $f(x)$ is reducible. If $g(x)$ is reducible, we may consider the following two cases: in the case that $d\ge 2e$,  by Proposition~\ref{prop1},  then $g(x)=(x^{2\ell}+\epsilon x^{\ell}+1)h_1(x)$, where $\epsilon \in \{-1,1\},\ \ell=\gcd(d,e)$ and $h_1(x)$ is irreducible; in the case that $d< 2e$, we consider the reciprocal polynomial $-g^*(x)=-x^dg(x^{-1})=x^d-x^{d-e}-1$ which is reducible and thus has the form $-x^dg(x^{-1})=(x^{2\ell}+\epsilon x^{\ell}+1)h_2(x)$ so that $g(x^{-1})=(1+\epsilon x^{-\ell}+x^{-2\ell})(-x^{d-2\ell}h_2(x))$. In both cases we have
$$g(x)=(x^{2\ell}+\epsilon x^{\ell}+1)h(x)$$
where $h(x)$ is irreducible by Proposition \ref{prop1}.  Since all zeros of $x^{2\ell}\pm x^{\ell}+1$ are the roots of unity, we have $h(\lambda^s)=0$. It follows that $h(x)|f(x)$. If $f(x)$ is also irreducible, then $f(x)=h(x)$. Hence
\begin{eqnarray*}
x^d+x^e-1 = (x^{2\ell}+\epsilon x^{\ell}+1)(x^a+x^b+x^c-1).
\end{eqnarray*}
It is easy to check that for both $\epsilon=1$ and $\epsilon=-1$, the above two sides  are always not equal when we set $x=1$, which is a  contradiction. Therefore $f(x)$ must be reducible.
\end{proof}

From the proof of the theorem, it can be seen that

\begin{Cor}\label{cor-a>d}
Under the same assumption as above. If 	${\mathcal{D}}(\boldsymbol{\alpha})\sim
{\mathcal{D}}(\boldsymbol{\beta})$, then $\max\{a,b,c\}\geq \max\{d,e\}$.
\end{Cor}

The following is a  sufficient condition for the irreducibility of quadrinomials.

\begin{Lem}[\cite{Du04,Lj60}]\label{prop2}
	Let $a,b,c$ be three distinct positive integers. If they are all
	odd, then the polynomial $f(x)=x^a+x^b+x^c\pm 1$ is irreducible over
	${\mathbb Q}$.
\end{Lem}

Suppose  $\text{gcd}(a,b,c)=2^km$, where $m$ is odd. Define
$$a'=a/{2^k},\ b'=b/{2^k}, \  c'=c/{2^k}$$ and
$$\bar{a}= \text{gcd}(a',b'-c'), \  \bar{b}= \text{gcd}(b',a'-c'),
\   \bar{c}= \text{gcd}(c',a'-b').$$  
We have a simple criterion of the irreducibility of quadrinomials.

\begin{Lem}[\cite{FiJo06}]\label{prop3}
	If $f(x)=x^a+x^b+x^c-1$, then  $f(x)$ is irreducible over ${\mathbb
		Q}$ if and only if $a'\ne 0\ (\text{mod}\ 2\bar{a}),\ b'\ne 0\
	(\text{mod}\ 2\bar{b}),\ c'\ne 0\ (\text{mod}\ 2\bar{c}).$
\end{Lem}

Following the above notation and combining Theorem \ref{mainthm1}, Lemmas \ref{prop2} and \ref{prop3}, we obtain an easy way  to verify that two Cantor sets are non-Lipschitz equivalent.

\begin{theorem}\label{thm2}
	Let $\boldsymbol{\alpha}=(\lambda^a,\lambda^b,\lambda^c)$ and
	$\boldsymbol{\beta}=(\lambda^d,\lambda^e)$ be two contraction
	vectors with $0< \lambda<1,\ a,b,c,d,e\in \Zp$. Then
	${\mathcal{D}}(\boldsymbol{\alpha})\nsim{\mathcal{D}}(\boldsymbol{\beta})$
	if any one of the following conditions holds:
	
	(1) $a,b,c$ are odd;
	
	(2) $a'\ne 0\ (\text{mod}\ 2\bar{a}),\ b'\ne 0\ (\text{mod}\
	2\bar{b}),\ c'\ne 0\ (\text{mod}\ 2\bar{c})$.
\end{theorem}

Let $a>b>c>0$ be integers; $\beta, \gamma,\delta\in \{-1,1\}$; and let $f(x)=x^a+\beta x^b+\gamma x^c+\delta$ be a  quadrinomial. It is shown  in \cite{Lj60} that $f(x)$ is reducible over $\mathbb{Q}$ if and only if $f(\eta)=0$ for some root of unity $\eta$. By using this, finally  we can prove our second main result.

\begin{theorem}\label{mainthm2}
Let $\boldsymbol{\alpha}=(\lambda^{a},\lambda^{b},\lambda^{c})$, $\boldsymbol{\beta}=(\lambda^{d},\lambda^{e})$ be two contraction vectors where $0<\lambda<1$,   $a=b+c$, $\gcd(a,b,c)=1$.  Then
$\mathcal D(\boldsymbol{\alpha})\sim\mathcal D(\boldsymbol{\beta})$  if and only if  $a=4$ with $\{b,c\}=\{1,3\}$ and $\{d,e\}=\{1,2\}$; or $a=8$ with $\{b,c\}=\{1,7\}$ and $\{d,e\}=\{1,5\}$, $\{2,3\}$.
\end{theorem}

\begin{proof}
First we prove the sufficient part. Iterating the $\lambda^2$ term in $({\lambda}^2,{\lambda})$, we obtain the contraction vector $({\lambda}^4,{\lambda}^3,{\lambda})$. Thus $\mathcal D({\lambda}^4,{\lambda}^3,{\lambda})\sim\mathcal D({\lambda}^2,{\lambda})$. Iterating the $\lambda$ term in $({\lambda}^8,{\lambda}^7,{\lambda})$, we obtain $\mathcal D({\lambda}^8,{\lambda}^7,{\lambda})\sim\mathcal D({\lambda}^8,{\lambda}^7,{\lambda}^9,{\lambda}^8,{\lambda}^2)$. Iterating the $\lambda^3$ term in $({\lambda}^3,{\lambda}^2)$ twice yields
\begin{equation*}
  \mathcal D ({\lambda}^3,{\lambda}^2) \sim \mathcal D({\lambda}^6,{\lambda}^5,{\lambda}^2) \sim  \mathcal D({\lambda}^9,{\lambda}^8,{\lambda}^8,{\lambda}^7,{\lambda}^2)
\end{equation*}
so that $\mathcal D({\lambda}^8,{\lambda}^7,{\lambda})\sim\mathcal D({\lambda}^3,{\lambda}^2)$. By Theorem~\ref{Raothm1}, we know that $\mathcal D({\lambda}^5,{\lambda})\sim\mathcal D({\lambda}^3,{\lambda}^2)$. Thus $\mathcal D({\lambda}^8,{\lambda}^7,{\lambda})\sim\mathcal D({\lambda}^5,{\lambda})$.

Now we prove the necessary part. From $a=b+c, \  \gcd(a,b,c)=1$, we know $b\ne c$. Without loss of generality we may assume $a>b>c$. If $d=e$,  by the proof of Theorem \ref{mainthm1}, then  $(a,b,c)$ should be $(1,2,2)$ or its permutations, contradicting the assumption of $a=b+c$. Hence we may assume $d>e$. It follows from Corollary \ref{cor-a>d} that $a\geq d$.

Define $f(x)=x^{a}+x^{b}+x^{c}-1$ and $g(x)=x^d+x^e-1$. If $\mathcal D(\boldsymbol{\alpha})\sim\mathcal D(\boldsymbol{\beta})$, then $g(x)$ and $f(x)$ have the same root $\lambda^{s}$ where $s$ is the common Hausdorff dimension. Moreover, by Theorem \ref{mainthm1},  $f(x)$ is reducible. Hence $f(x)=A(x)B(x)$ as in Proposition \ref{prop3.2} where every root of $A(x)$ and no root of $B(x)$ is a root of unity, $B(x)$ is irreducible except when $f(x)$ is one of the four forms \eqref{3.1}-\eqref{3.4}.

\bigskip
	
{\it\textbf{Case 1.}}  Suppose $B(x)$ is reducible. From $\gcd(a,b,c)=1$ and Proposition~\ref{prop3.2}, we have $$f(x)=x^{8}+x^{7}+x-1=(x^{2}+1)(x^{3}+x^{2}-1)(x^{3}-x+1).$$
If $g(x)$ is irreducible, then $g(x)=x^{3}+x^{2}-1$, so we have $(a,b,c,d,e)=(8,7,1,3,2)$. If $g(x)$ is reducible, then by Proposition \ref{prop1},  $g(x)=(x^{2\ell}+\epsilon x^\ell+1)h(x)$, where $\epsilon \in \{-1,1\}$, $\ell=\gcd(d,e)$ and $h(x)$ is irreducible. Since all zeros of $x^{2\ell}\pm x^\ell+1$ are the roots of unity, we have $h(\lambda^s)=0$. It follows that $h(x)|f(x)$. Since $x^2+1$ and $x^{3}-x+1$ have no roots in $(0,1)$, we have $h(x)|x^{3}+x^{2}-1$. Hence $h(x)=x^{3}+x^{2}-1$ by  the irreducibility of $x^{3}+x^{2}-1$ over the rationals. Therefore,
	\begin{align*}
		x^d+x^e-1 = (x^{2\ell}+\epsilon x^\ell+1)(x^{3}+x^{2}-1).
	\end{align*}
By letting $x=1$, it can be easily seen that $\epsilon=-1$. This yields $d=3+2\ell$ and
$$x^e =x^{2+2\ell}-x^{2\ell}-x^{3+\ell}-x^{2+\ell}+x^{\ell}+x^{3}+x^{2}.$$
It follows that $\ell=1$, $e=1$, $d=5$, then we have $(a,b,c,d,e)=(8,7,1,5,1)$.

\bigskip

{\it\textbf{Case 2.}}  Suppose  $B(x)$ is irreducible. Let $t_1=\gcd(a, b-c)$, $t_2=b$ and $t_3=c$. From $a=b+c$, $\gcd(a,b,c)=1$ and Proposition~\ref{prop3.3}, we know that all possible roots of unity of $f(x)$ are simple zeros, which are to be found among the zeros of
	\begin{align*}
		x^{t_1}=\pm 1; \quad x^{b}=\pm 1; \quad x^{c}=\pm 1.
	\end{align*}

Let $\eta$ be a root of unity of $f(x)$. Since $a=b+c$, we have $f(x)=(x^b+1)(x^c+1)-2$.  If $\eta$ is the zero of $x^{b}=1$, from $f(\eta)=(\eta^b+1)(\eta^c+1)-2=0$, it follows that  $\eta^c=0$, that is impossible;  if $\eta$ is the zero of $x^{b}=-1$, then $f(\eta)=(\eta^b+1)(\eta^c+1)-2=-2\ne 0$, that is a contradiction. Similarly, $\eta$ cannot be the zero of $x^{c}=\pm1$. So all the roots of unity  of $f(x)$ can only be found in the zeros of  $x^{t_1}=\pm1$.

Let $p(x)=(x^{t_1}-1)(x^{t_1}+1)$.  Then $A(x)|p(x)$ as all possible roots of unity of $A(x)$ (or $f(x)$) are simple zeros.

By the assumptions that $a=b+c$ and  $\gcd(a,b,c)=1$, we have $\gcd(b,c)=1$, and $t_1=\gcd(a, b-c)=\gcd(b+c, b-c)$.  If $b, c$ are different from odevity, then $t_1=1$. If $b, c$ are both odd, then $t_1=2$. If $t_1=1$, then $p(x)=(x-1)(x+1)$. If $t_1=2$, then $p(x)=(x^2-1)(x^2+1)=(x-1)(x+1)(x^2+1)$. Due to  $A(x)|p(x)$ and the fact $f(1)\ne 0$ (hence $A(1)\ne 0$),  all possible forms of $A(x)$ could be
$$x+1;\quad  x^2+1;\quad (x+1)(x^2+1).$$

\bigskip

{\it\textbf{Case 2.1.}} If $A(x)=x+1$, then $f(x)=A(x)B(x)=(x+1)B(x)$. Assume that $g(x)$ is irreducible. Since $g(x)$ and $f(x)$ have the same zero $\lambda^{s}$, we have $g(x)=B(x)$ so that
\begin{align*}
x^a+x^b+x^c-1&=(x+1)(x^d+x^e-1)\\
& =x^{1+d}+x^{1+e}-x+x^d+x^e-1.
\end{align*}
By comparing the powers of $x$ on both sides,  if $e<c$, then $e=1$, $a=d+1$ and $b+c=1+e+d=d+2$,  contradicting the assumption of $a=b+c$; if $e=c$,  then  $x$ on the right side cannot be canceled by any other term on the both sides; if $e>c$, then $x^c$ can not be canceled by any other term on the both sides. All are impossible.  Assume that  $g(x)$ is reducible. By Proposition \ref{prop1},  $g(x)=(x^{2\ell}+\epsilon x^\ell+1)h(x)$, where $\epsilon \in \{-1,1\}$, $\ell=\gcd(d,e)$ and $h(x)$ is irreducible. Notice that $h(\lambda^s)=B(\lambda^s)=0$. Thus $h(x)=B(x)$ so that
$$(x^a+x^b+x^c-1)(x^{2\ell}+\epsilon x^\ell+1)=(x+1)(x^d+x^e-1).$$
Since $a\ge d$ and $\ell=\gcd(d, e)\ge 1$, we have $a+2\ell>1+d$. Contradiction.

\bigskip

{\it\textbf{Case 2.2.}} If $A(x)=x^2+1$, then $f(x)=(x^2+1)B(x)$. If $g(x)$ is irreducible, then  $g(x)=B(x)$ so that
	\begin{align}
x^a+x^b+x^c-1 = x^{2+d}+x^{2+e}-x^2+x^d+x^e-1. \label{identity of case2.1}
	\end{align}
If $e<c$, then $e=2, a=d+2$ and $b+c=2+e+d=d+4$, contradicting the assumption of $a=b+c$. If $e>c$, then $x^c$ on the left side of \eqref{identity of case2.1} cannot be canceled by any other term on the both sides of \eqref{identity of case2.1}. That is impossible.
If $e=c$, then \eqref{identity of case2.1} can be reduced into
$$x^a+x^b=x^{2+d}+x^{2+e}-x^2+x^d.$$  Hence $d=2, a=d+2=4$ and $b=2+c$. From $a=b+c$, it follows that $c=1$.  Thus $(a,b,c,d,e)=(4,3,1,2,1)$.  Assume that $g(x)$ is reducible. Similarly as in Case~2.1, there exists $\epsilon\in\{-1,1\}$ such that
  $$(x^a+x^b+x^c-1)(x^{2\ell}+\epsilon x^\ell+1)=(x^2+1)(x^d+x^e-1),$$
 where $\ell=\gcd(d, e)$. Since $a\ge d$ and $\ell\ge 1$, we have $a+2\ell>2+d$. This is impossible.

\bigskip

{\it\textbf{Case 2.3.}} If $A(x)=(x+1)(x^2+1)$, then $f(x)=(x+1)(x^2+1)B(x)$. If $g(x)$ is irreducible, we have $g(x)=B(x)$ so that
\begin{align*}
x^a+x^b+x^c-1=(x+1)(x^2+1)(x^d+x^e-1)
\end{align*}
Let $x=1$, then the left side is $2$ while the right side is $4$. That is impossible. Assume that $g(x)$ is reducible. Similarly as in Case~2.1, there exists $\epsilon\in\{-1,1\}$ such that $$(x^a+x^b+x^c-1)(x^{2\ell}+\epsilon x^\ell+1)=(x+1)(x^2+1)(x^d+x^e-1),$$
where $\ell=\gcd(d, e)$. Let $x=1$, then $2(2+\epsilon)=4$ that implies $\epsilon=0$, which contradicts that  $\epsilon \in \{-1,1\}$.
\end{proof}

We remark that Theorem \ref{mainthm2} might be true if we remove the assumption of $\gcd(a,b,c)=1$ by using Theorem \ref{thm2}, but the proof will become tedious. As for the exceptional form  \eqref{3.3} not included in the theorem, we have the following complement.

\begin{Prop}
	Let $0< \lambda<1$. Then ${\mathcal{D}}(\lambda^8,\lambda^4,\lambda^2)\sim {\mathcal{D}}(\lambda^3,\lambda^2)$.
\end{Prop}

\begin{proof}
	Let $f_1(x)=\lambda^8 x, f_2(x)=\lambda^4 x+\lambda^3-\lambda^4$ and $f_3(x)=\lambda^3x+1-\lambda^2$ be an IFS on $\mathbb{R}$ and let $E$ be the associated self-similar set; let  $g_1(x)=\lambda^3 x, g_2(x)=\lambda^2 x+1-\lambda^2$ be another IFS on $\mathbb{R}$ and let $F$ be the associated self-similar set. Obviously, both $E$ and $F$ are dust-like and $E\in {\mathcal{D}}(\boldsymbol{\alpha}), F\in {\mathcal{D}}(\boldsymbol{\beta})$. Hence  in order to show ${\mathcal{D}}(\boldsymbol{\alpha})\sim {\mathcal{D}}(\boldsymbol{\beta})$,  it suffices to show $E\sim F$. Indeed, we can find $E$ and $F$ have the same graph-directed structure in the following way: let $E_1=E; E_2=\lambda^5 E\cup (\lambda E + 1-\lambda)$, then
	\begin{eqnarray*}
		E_1 &=& (\lambda^2 E_1+1-\lambda^2)\cup (\lambda^3E_2); \\
		E_2 &=& (\lambda^5E_1)\cup(\lambda^3 E_1+1-\lambda^3)\cup (\lambda^4 E_2+1-\lambda).
	\end{eqnarray*}
	Similarly for $F$, we let $F_1=F_2=F$, then
	\begin{eqnarray*}
		F_1 &=& (\lambda^2 F_1+1-\lambda^2)\cup (\lambda^3F_2); \\
		F_2 &=& (\lambda^3F_1)\cup(\lambda^5 F_1+1-\lambda^2)\cup (\lambda^4 F_2+1-\lambda^4).
	\end{eqnarray*}
	Both $\{E_1, E_2\}$ and $\{F_1, F_2\}$ are dust-like graph-directed sets and satisfy the conditions of Theorem 2.1 of \cite{RaRuXi06}, thus $E_i\sim F_i$ for $i=1,2$. Therefore $E\sim F$.
	
    We also give an  alternative proof. Iterating the terms in $(\lambda^5, \lambda)$ yields 
	$${\mathcal{D}}(\lambda^5,\lambda)\sim {\mathcal{D}}(\lambda^{10},\lambda^6,\lambda^6, \lambda^2)\sim {\mathcal{D}}(\lambda^{10},\lambda^6,\lambda^6, \lambda^7, \lambda^3)\sim {\mathcal{D}}(\lambda^{10},\lambda^6,\lambda^6, \lambda^{12},\lambda^8,\lambda^8,\lambda^4).$$ Iterating $\lambda^4$ and $\lambda^2$ in $(\lambda^8,\lambda^4,\lambda^2)$ yields 
	$${\mathcal{D}}(\lambda^8,\lambda^4,\lambda^2)\sim{\mathcal{D}}(\lambda^8,\lambda^{12},\lambda^8,\lambda^6,\lambda^{10},\lambda^6,\lambda^4).$$ It follows from Theorem \ref{Raothm1} that  ${\mathcal{D}}(\lambda^8,\lambda^4,\lambda^2)\sim {\mathcal{D}}(\lambda^5,\lambda)\sim {\mathcal{D}}(\lambda^3,\lambda^2).$
\end{proof}
\end{section}

\bigskip

\begin{section}{Cantor sets with homogeneous contraction vectors}

\begin{Lem}\label{lemma on existence of a cut-set}
Let $\boldsymbol{\alpha}=({\alpha},\dots,{\alpha})\in
{\mathbb{R}}^m$, $E\in {\mathcal{D}}(\boldsymbol{\alpha})$ with
$\dim_H E=s$. If there is a sequence of  positive integers
$\{a_i\}_{i=1}^n$ satisfying $\sum_{i=1}^n{\alpha}^{sa_i}=1$. Then
there exists a partition $\Lambda=\{{\bf i}_i: i=1,\dots, n\}$ for
$E$ such that $\alpha_{{\bf i}_i}={\alpha}^{a_i}$.
\end{Lem}

\begin{proof}
Let $\ell=\max_{1\leq i\leq n} a_i$. We shall show the lemma by
induction. If $\ell=1$, then $n=m$ and $\Lambda=\Sigma$ is a
partition.  If $\ell=k$, the lemma is true, for $\ell=k+1$, write
$$\Lambda_1=\{i: a_i=k+1\} \quad\text{and}\quad\Lambda_2=\{i:a_i\leq
k\}.$$ Then
$$
1 = \sum_{i=1}^n{\alpha}^{sa_i} = \sum_{i=1}^nm^{-a_i}= \sum_{i\in
\Lambda_1}m^{-a_i}+\sum_{i\in\Lambda_2}m^{-a_i} = \#\Lambda_1
m^{-(k+1)}+\sum_{i\in \Lambda_2}m^{-a_i}
$$ which implies $$\#\Lambda_1=m(m^k-\sum_{i\in
\Lambda_2}m^{k-a_i}).$$ Let $r=(m^k-\sum_{i\in
\Lambda_2}m^{k-a_i})$, hence $$1= rm^{-k}+\sum_{i\in
\Lambda_2}m^{-a_i}.$$ By the assumption, there exists
a partition $\Lambda=\{{\bf j}_i: 1\le i \le  r\}\cup\{{\bf i}_i: i\in\Lambda_2\}$ such that $\alpha_{{\bf j}_i}={\alpha}^k, i=1,\dots, r$  and $\alpha_{{\bf i}_i}={\alpha}^{a_i}, \  i\in\Lambda_2$. Take $$\Lambda'=(\cup_{i=1}^r{\bf j}_i\Sigma)\cup\{{\bf i}_i: i\in\Lambda_2\}.$$ That is a partition as desired. Therefore  the result follows.
\end{proof}

The next lemma is a special case of \cite{RaRuXi06}:

\begin{Lem} \label{lip. equ. lemma on cut-sets}
Let $E\in {\mathcal{D}}(\boldsymbol{\alpha})$ and $F\in {\mathcal{D}}(\boldsymbol{\beta})$. If there exist two partitions $\Lambda_1=\{{\alpha}_{{\bf i}_k}\}_{k=1}^N, \Lambda_2=\{{\beta}_{{\bf j}_k}\}_{k=1}^N$ for $E, F$, respectively, such that $({\alpha}_{{\bf i}_1},\dots, {\alpha}_{{\bf i}_N})$ is a permutation of $({\beta}_{{\bf j}_1},\dots, {\beta}_{{\bf j}_N})$. Then $\boldsymbol{\alpha}\sim \boldsymbol{\beta}$, hence $E\simeq F$.
\end{Lem}

\begin{theorem}
Let $\boldsymbol{\alpha}=({\alpha},\dots,{\alpha})\in {\mathbb{R}}^m,\  \boldsymbol{\beta}=({\beta}_1,\dots,{\beta}_n)\in {\mathbb{R}}^n$ be two contraction vectors. Then ${\mathcal{D}}(\boldsymbol{\alpha})\sim {\mathcal{D}}(\boldsymbol{\beta})$ if and only if $\boldsymbol{\alpha} \sim \boldsymbol{\beta}$.
\end{theorem}

\begin{proof}
The  sufficient part is obvious by Lemma \ref{lip. equ. lemma on cut-sets}. We only need to prove the necessary part. Let $E\in {\mathcal{D}}(\boldsymbol{\alpha})$ and $F\in{\mathcal{D}}(\boldsymbol{\beta})$. If $E\sim F$, then they have the same Hausdorff dimension $s$. Moreover, by Theorem \ref{Falconer necessary condition on lip.equ.}, there exists $q_0\in \Zp$ such that
$$\text{sgp}({\beta}_1^{q_0},\dots,{\beta}_n^{q_0})\subset
\text{sgp}(\alpha,\dots,\alpha)=\{\alpha^p: p\in \Zp\}.$$
Thus there exists $p_j\in \Zp$ such that
$\beta_j=\alpha^{\frac{p_j}{q_0}}$ for all $j.$
Let $d=\text{gcd}(q_0,p_1,\dots,p_n)$ and $q=q_0/d$, $p_j'=p_j/d$ for all $j=1,\ldots,n$. Then
$\text{gcd}(q,p_1',\dots,p_n')=1$. Since $${\mathbb
Q}(\beta_1^s,\dots,\beta_n^s)={\mathbb
Q}(\alpha^s,\dots,\alpha^s)={\mathbb Q}(\frac{1}{m})={\mathbb Q},$$
We have $\beta_j^s\in {\mathbb Q}$, i.e.,
$\alpha^{p_js/q_0}=m^{-p_j'/q}\in {\mathbb Q}$ for all $j$. Combining this with $\mathrm{gcd}(q,p_1',\dots,p_n')=1$, we know that $m^{1/q}\in
{\mathbb{Q}}$ so that $m^{1/q}\in\Zp$.

Let $k=m^{1/q}$ and $\lambda= \alpha^{1/q}$.
Then $k\lambda^s=1$. Define another contraction vector
$\boldsymbol{\lambda}=(\lambda,\dots,\lambda)\in {\mathbb{R}}^k$.
Let $E_0\in {\mathcal D}(\boldsymbol{\lambda})$. Then $\dim_H
E_0=s$. Notice that $\Lambda=\{1,\dots,k\}^q$ is a partition for $E_0$
with $\#{\Lambda}=k^q=m$ and $\lambda_{\bf i}=\lambda^q=\alpha$ for
${\bf i}\in \Lambda$. Thus, by Lemmas \ref{lemma on existence of a
cut-set} and \ref{lip. equ. lemma on cut-sets}, we have $\boldsymbol{\alpha}\sim \boldsymbol{\lambda}$.
Similarly, from $\beta_j=\alpha^{p_j/q}=\lambda^{p_j}$ and
$\sum_{j=1}^n \lambda^{sp_j}=\sum_{j=1}^n\beta_j^s=1$, we have
$\boldsymbol{\lambda}\sim  \boldsymbol{\beta}$ as well. Therefore $\boldsymbol{\alpha}\sim \boldsymbol{\beta}$.
\end{proof}

The theorem easily yields the following useful result, which can  also be induced by Theorem \ref{Falconer necessary condition on lip.equ.}.

\begin{Cor}
Let $\boldsymbol{\alpha}=(\alpha,\dots,\alpha)\in {\mathbb{R}}^m,\ \boldsymbol{\beta}=(\beta,\dots,\beta)\in {\mathbb{R}}^n$, and $m{\alpha}^s=n{\beta}^s=1$. Then ${\mathcal{D}}(\boldsymbol{\alpha})\sim{\mathcal{D}}(\boldsymbol{\beta})$ if and only if $\frac{\log m}{\log n}\in {\mathbb{Q}}.$
\end{Cor}
\end{section}

\bigskip

\noindent {\bf Acknowledgements:} The first author gratefully acknowledges the support of K. C. Wong Education Foundation and DAAD.

\end{document}